\documentclass[3p,sort&compress,number,preprint]{elsarticle}

\usepackage{amsmath,amsthm,amssymb}

\newtheorem{theorem}{Theorem}
\newtheorem{lemma}{Lemma}

\newtheorem{corollary}{Corollary}

\newtheorem{remark}{Remark}

\usepackage{hyperref}
\usepackage{booktabs}
\usepackage{tikz}
\usetikzlibrary{calc}
\usetikzlibrary{matrix}

\DeclareMathOperator{\im}{ran}
\DeclareMathOperator*{\spanof}{span}

\def\divdifsymb{\scaleobj{1.2}{\mathord{\kern.34em\vrule width.6pt height6.3pt depth-.28pt \kern-.34em\Delta}}}

\newcommand{\setof}[1]{\left\{{#1}\right\}}
\newcommand{\cset}[2]{\setof{#1\,:\,#2}}
\newcommand{\ivcc}[2]{\left[#1,#2\right]}

\newcommand{\ivu}{\ivcc{0}{1}}

\newcommand{\abs}[1]{\left|#1\right|}
\newcommand{\norm}[1]{\left\Vert#1\right\Vert}

\newcommand{\normop}[1]{\left\Vert#1\right\Vert_{op}}

\newcommand{\spacecf}{C(\ivcc{0}{1})}

\newcommand{\ballo}[2]{B({#1},{#2})}

\newcommand{\uballo}{\ballo{0}{1}}
\newcommand{\uballc}{\overline{\uballo}}

\newcommand{\Nn}{{\mathbb N}}

\newcommand{\Cc}{{\mathbb C}}

\newcommand{\id}{I}

\newcommand{\boundop}[2]{\mathcal{L}(#1,#2)}
\newcommand{\linope}[1]{\mathcal{L}(#1)}
\newcommand{\boundope}[1]{\mathcal{L}(#1)}
\newcommand{\compop}[2]{\mathcal{K}(#1,#2)}
\newcommand{\compope}[1]{\mathcal{K}(#1)}

\newcommand{\fredop}[2]{\Phi(#1,#2)}
\newcommand{\fredope}[1]{\Phi(#1)}

\newcommand{\fredmope}[1]{\Phi_{-}(#1)}

\newcommand{\weylope}[1]{\mathcal{W}(#1)}

\newcommand{\browdope}[1]{\mathcal{W}_B(#1)}

\newcommand{\ind}[1]{\mathrm{ind}(#1)}
\newcommand{\asc}[1]{\mathrm{asc}(#1)}
\newcommand{\dsc}[1]{\mathrm{dsc}(#1)}

\newcommand{\spec}[1]{\sigma(#1)}
\newcommand{\specp}[1]{\sigma_p(#1)}

\newcommand{\resolventset}[1]{\rho(#1)}
\newcommand{\resolventop}[2]{R(#1,#2)}

\newcommand{\ie}{i.\,e.}

\newcommand{\eg}{e.\,g.}

\journal{}
\usepackage{fancyhdr}
\renewcommand{\headrulewidth}{0.0pt}

\begin{document}

\begin{frontmatter}
  \title{Iterates of Markov operators and their limits}

  \author{Johannes~Nagler}
  \ead{johannes.nagler@uni-passau.de}

  \address{Fakult\"at f\"ur Informatik und Mathematik, Universit\"at Passau, Germany}

  \begin{abstract}
    It is well known that iterates of quasi-compact operators converge towards
    a spectral projection, whereas the explicit construction of the limiting
    operator is in general hard to obtain. Here, we show a simple method to
    explicitly construct this projection operator, provided that the fixed
    points of the operator and its adjoint are known which is often the case for
    operators used in approximation theory.

    We use an approach related to Riesz-Schauder and Fredholm theory to analyze
    the iterates of operators on general Banach spaces, while our main result
    remains applicable without specific knowledge on the underlying framework.
    Applications for Markov operators on the space of continuous functions
    $C(X)$ are provided, where $X$ is a compact Hausdorff space.
  \end{abstract}

  
\end{frontmatter}
\fancypagestyle{pprintTitle}{%
\lhead{} \chead{}\rhead{}
\lfoot{}\cfoot{}\rfoot{{\footnotesize\itshape \hfill Preprint, \today}}
\renewcommand{\headrulewidth}{0.0pt}
}


The behaviour of the iterates of Markov operators has been studied extensively
in modern ergodic theory, while in general the limiting operator is not
explicitly given. A comprehensive overview on limit theorems for quasi-compact
Markov operators can be found in \citet{hennion2001}. In this article, we
construct the limit of the iterates of quasi-compact operators that satisfy a
spectral condition. It will be shown under which conditions the limit exists and
how the limiting projection operator can be explicitly constructed using the
inverse of a Gram matrix. The explicit knowledge of the limiting operator is of
interest in many applications. 

This research is motivated by studying general Markov operators on the space of
continuous functions $C(X)$, where $X$ is a compact Hausdorff space.
\citet{lotz1981} has already shown uniform ergodic theorems for Markov operators
on $C(X)$. For specific classes of operators, the limiting operator has been
provided as shown for instance by \citet{kelisky1967}, \citet{karlin1970} and
\citeauthor{gavrea2011} \cite{gavrea2011,gavrea2011b}. Recently,
\citet{altomare2013} has shown a different approach using the concept of
Choquet-boundaries and results from Korovkin-type approximation theory.
\citet{altomare2014} have shown an application where they discussed differential
operators associated with Markov operators, where also the knowledge of limit of
the iterates is significant. Another application has been shown in the field of
approximation theory, where the iterates can be used to prove lower estimates
for Markov operators with sufficient smooth range, see \citet{nagler2015joc}.

It is worthwhile to mention that in most methods the limiting operator has to be
known apriori. Here, we show an elegant extension to general Banach spaces for
quasi-compact Markov operators. This extension provides a very general framework
to explicitly construct the limiting operator with a simple method without prior
knowledge of this operator. 



After an introductory example, we introduce briefly our notation and recall the
most important results that are necessary to prove our results. All of these
results are well-known and can be found, \eg, in the classical books of
\citet{ruston1986}, \citet{rudin1991}, \citet{heuser1982}. In the next section,
we discuss how the complemented subspace for some finite-dimensional eigenspace
of an operator can be expressed in terms of the corresponding projection. We
will start using the standard coordinate map to show the principle of our
approach. Using a generalized version of the coordinate map we show conditions
when the coordinate map on some eigenspace can be expressed in terms of a basis
for this eigenspace and a basis of the corresponding eigenspace of the adjoint
operator. These results are used to prove the limiting behaviour of the iterates
of quasi-compact Markov operators. 

\section{An introductory example}
We now demonstrate the simplicity of our result in a short example on
$\spacecf$, the space of continuous functions on the
interval $\ivu$. Thereby, let $n$ be a positive integer and suppose that
$\setof{x_j}_{j=1}^n$ form a partition of $\ivu$, i.e. $0 = t_1 < t_2 < \cdots < t_n = 1$.
We consider the positive finite-rank operator $T:\spacecf \to \spacecf$, 
defined for $f \in \spacecf$ by
\begin{equation}
  \label{eq:discrete_operator}
  Tf = \sum_{k=1}^n f(t_{k})p_k,
\end{equation}
where $p_1,\ldots,p_n \in \spacecf$ are positive functions that form a
partition of unity, i.e., $\sum_{k=1}^n p_k(t) = 1$ for all $t \in \ivu$. It is
easy to see that in this case $T1 = 1$ and $\normop{T} = r(T) = 1$, where $r(T)$
is the spectral radius of $T$. Besides, we assume that 
\begin{equation*}
  \sum_{k=1}^n t_{k} p_k(t) = t, \qquad t\in \ivu,
\end{equation*}
i.e., $Tf = f$ holds whenever $f$ is a linear function. From that it follows
already that $p_1(0) = p_n(1)= 1$ and $T$ interpolates at $0$ and $1$, as
\begin{equation*}
  Tf(0) = \sum_{k=1}^n f(t_k)p_k(0) = f(t_1) = f(0), \quad  Tf(1) = \sum_{k=1}^n f(t_k)p_k(1) = f(t_n) = f(1).
\end{equation*}
The introduced operator is a Markov operator, as it is a positive contraction
and $T1 = 1$ holds. Two fixed points for $T^*$ are given due to the
interpolation at $0$ and $1$. If $\delta_0,\delta_1$ denote the continuous
functionals that evaluate continuous functions at $0$ and $1$ respectively, then
$\delta_0(Tf) = \delta_0(f)$ and $\delta_1(Tf) = \delta_1(f)$ holds for all $f
\in \spacecf$.

In the following, we want to answer the question whether the limit of the
iterates $T^m$ for $m \to \infty$ exists and if so to which operator the
iterates converge. In \citet{nagler2014fr} is has been shown that the partition
of unity property, which is here equivalent to the ability to reproduce constant
functions, guarantees that $\sigma(T) \subset \uballo \cup \setof{1}$. To apply
our main result, we have to specify the fixed point spaces of $T$ and its
adjoint $T^*$. Using the partition of unity property of $T$ and the ability of
$T$ to reproduce linear functions as well as the ability to interpolate at the
endpoints of the interval $\ivu$, we derive the fixed-point spaces
\begin{align*}
  \ker(T-I) &= \cset{f \in \spacecf}{Tf = f} = \spanof(1, x),\\
  \ker(T^*-I) &= \cset{\alpha^* \in \spacecf^*}{\alpha^*(Tf) = \alpha^*(f) \text{ for all } f \in \spacecf} = \spanof(\delta_0, \delta_1).
\end{align*}
Then we consider the Gram matrix
\begin{equation*}
  G :=
  \begin{pmatrix}
    \delta_0(1) & \delta_0(x)\\
    \delta_1(1) & \delta_1(x)
  \end{pmatrix} = 
  \begin{pmatrix}
    1 & 0 \\
    1 & 1
  \end{pmatrix},
\end{equation*}
where the functionals of $\ker(T^* - I)$ operate on the fixpoints of $T$.
Indeed, this matrix is invertible with
\begin{equation*}
  A := G^{-1} =  
  \begin{pmatrix}
    1 & 0 \\
    -1 & 1
  \end{pmatrix},
\end{equation*}    
and we are able to use the coefficients $a_{11} = 1$, $a_{12} = 0$, $a_{11} = -1$, $a_{12} = 1$ 
to conclude by \autoref{thm:iterates_limit} that
\begin{equation*}
  \lim_{m \to \infty} \normop{T^m - P} = 0,
\end{equation*}
where the finite-rank projection $P:\spacecf \to \ker(T-\id)$ is defined for $f \in \spacecf$ by
\begin{align*}
  Pf &= \left(a_{11}\delta_0(f) + a_{12}\delta_1(f)\right)\cdot 1 
          + \left(a_{21}\delta_1(f) + a_{22}\delta_0(f)\right)\cdot x\\
       &= \delta_0(f)\cdot 1 + \delta_1(f) - \delta_0(f)\cdot x = f(0) + \left(f(1) - f(0)\right)x.
\end{align*}
The iterates converge to the linear interpolation operator that
interpolates at the endpoints of $\ivu$.  In this example we
demonstrated the underlying framework for finite-rank operators that
reproduce constant and linear functions. Operators of this kind are,
\eg, the Bernstein and the Schoenberg operator that are often used in
CAGD and approximation theory.  However, the convergence is guaranteed
for all quasi-compact Markov operators. Note that the following implications hold:
 \begin{equation*}
   \text{finite-rank}\quad\Rightarrow\quad\text{compact}\quad
   \Rightarrow\quad\text{Riesz}\quad\Rightarrow\quad\text{quasi-compact}
 \end{equation*}
 
\section{Notation}
For the convenience of the reader this section provides not only the used
notation throughout this article but also a compact overview over the most
important facts that are used later. All results in this chapter can be found in
the comprehensive books of \citet{heuser1982}, \citet{rudin1991} and
\citet{ruston1986}.

The general setting considers $X$ as a complex Banach space equipped with a norm
$\norm{\cdot}_X$. If the used norm is unambiguous we will just use the
abbreviated version $\norm{\cdot}$. Note that the results shown here are also
applicable on real Banach spaces using a standard complexification scheme as
outlined, e.g, in \citet[pp.~7--16]{ruston1986}.

The Banach algebra of bounded linear operators on $X$ is denoted by
$\boundope{X}$ equipped with the usual operator norm $\normop{\cdot}$. The
identity operator on $X$ is $\id \in \linope{X}$. The corresponding topological
dual space $\boundop{X}{\Cc}$ is denoted by $(X^*, \norm{\cdot}_{X^*})$. The
\emph{range} and \emph{null space} of $T \in \boundope{X}$ is denoted by
$\im(T)$ and $\ker(T)$, respectively. The \emph{closure} of $M \subset X$ is
denoted by $\overline{M}$. We denote the space of all compact operators from $X$
to $Y$ by $\compop{X}{Y}$.

\subsection{Annihilators}
For $M \subset X$ and $\Lambda \subset X^*$. 
we denote by $M^\bot$ the \emph{annihilator} of $M$, \ie
\begin{equation*}
  M^\bot := \cset{x^* \in X^*}{x^*(x) = 0 \text{ for every } x \in M} \subset X^*,
\end{equation*}
and by $\Lambda_\bot$ the \emph{pre-annihilator} of the set $\Lambda$, \ie
\begin{equation*}
  \Lambda_\bot = \cset{x \in X}{x^*(x) = 0 \text{ for every } x^* \in \Lambda} \subset X.
\end{equation*}
Recall that if \(X\) and \(Y\) are Banach spaces and \(T \in \boundop{X}{Y}\),
then
\begin{equation}
  \label{eq:prel_relation_annihilator_nullspace}
  \ker(T^*) = T(X)^\bot, \quad \ker(T) = T^*(Y^*)_\bot,
\end{equation}
where $T^*$ denotes the \emph{adjoint} of $T$.

\subsection{Fredholm, Weyl and Browder operators}
\label{ssec:compact_fredholm_riesz}
An operator $T \in \linope{X}$ is said to have \emph{finite ascent} if there
exists \(k \in \Nn\) such that $\ker(T^k) = \ker(T^{k+1})$. The smallest integer
with this property is the \emph{ascent} of $T$ and will be denoted by $\asc{T}$.
Accordingly, $T$ has \emph{finite descent} if there exist $k \in \Nn$ such that
$\im(T^k) = \im(T^{k+1})$ and we denote by $\dsc{T}$ the smallest integer
with this property and call this number the \emph{descent} of $T$. Recall, that
\begin{equation}
  \label{eq:finite_asc-desc_properties}
  \asc{T} \leq m < \infty\quad \text{iff}\quad \ker(T^n) \cap \im(T^m) = \setof{0}
\end{equation}
holds, where $n > 0$ is arbitrary.
If both, the ascent and the descent, are finite, then they are
equal. In this case, the operator $T$ is said to have \emph{finite chain length} $p$
and yields a direct sum decomposition in the following way:
  \begin{equation}
    X = \im(T^p) \oplus \ker(T^p).
  \end{equation}
We will ask later when this space decomposition can be derived by so called
spectral projections.

We denote by $\alpha(T) := \dim(\ker(T))$ the \emph{nullity} of $T$ and by 
$\beta(T) := \dim(\ker(T^*))$ the \emph{deficiency} of the operator $T$.
We denote by $\fredop{X}{Y}$ the set of Fredholm operators, i.e. all operators
where the nullity and the deficiency are finite, and by $\ind{T} = \alpha(T) -
\beta(T)$ we denote the \emph{index} of $T$.

Now we can relate the concept of Fredholm operators, \ie, the nullity and the
deficiency, with the concept of the ascent and descent. Recall that if \(\asc{T}
< \infty\), then \(\alpha(T) \leq \beta(T)\) and if \(\dsc{T} < \infty\), then
\(\alpha(T) \geq \beta(T)\) holds.
According to this relation, we identify all Fredholm operators $T \in
\fredope{X}$ with finite ascent, $p = \asc{T} < \infty$, as operators where
\begin{equation*}
  \dim(\ker(T)) \leq \dim(\ker(T^*)) < \infty.
\end{equation*} 
We will denote all such linear operators $T$ that have a Fredholm index less or
equal than zero by the set $\fredmope{X}$. 
A bounded operator $T \in \linope{X}$ is said to be a \emph{Weyl operator} if
$T \in \fredope{X}$ with index $0$. The class of all Weyl operators on
$X$ will be denoted by $\weylope{X}$.

\begin{table}[tb]
  \centering
  \begin{tabular}[t]{lccc}
    \toprule
                & $\im(T)$ closed & $\alpha(T) = \beta(T) < \infty$ & $\asc{T} = \dsc{T} < \infty$\\ \midrule
    $T$ Fredholm & yes     & not necessarily & not necessarily \\
    $T$ Weyl     & yes     & yes             & not necessarily \\
    $T$ Browder  & yes     & yes             & yes              \\
    \bottomrule
  \end{tabular}
  \caption{Comparison between Fredholm, Weyl and Browder operators.}
  \label{tab:browder_overview}
\end{table}

A bounded operator $T \in \linope{X}$ is said to be a \emph{Browder operator} if
it is a Fredholm operator with finite chain length. We will denote the sets of
all Browder operators on $X$ by $\browdope{X}$. Each Browder operator $T$ is in
fact a Weyl operator, as in that case $\asc{T} = \dsc{T} < \infty$ and
$\alpha(T) = \beta(T) < \infty$ holds.
A comparison between both classes is shown in \autoref{tab:browder_overview}.
Note that due to the finite chain-length $p$ of a Browder
operator $T$, the following properties hold:
\begin{enumerate}
\item $\dim(\ker(T)) = \dim(\ker(T^*)) < \infty$,
\item $X = \ker(T^p) \oplus \im(T^p)$.
\end{enumerate}
In this article, we are interested in operators $T \in \linope{X}$ where $(T -
\lambda I)$ is a Browder operator. In that case, we construct an explicit
projection for the space decomposition shown in the second item.

\subsection{Spectral projections}
We denote by $\resolventset{T}$ the \emph{resolvent set} of $T \in
\boundope{X}$. The \emph{resolvent} of $T$ corresponding to $\lambda \in \Cc$
will be denoted by $\resolventop{T}{\lambda} := (T - \lambda I)^{-1}$. The
\emph{spectrum} of $T$ is denoted by $\spec{T}$, the spectral radius by $r(T)$.

Using functional calculus, it is well known that spectral projections exactly
provide the space decomposition discussed previously. The \emph{spectral projection}
associated to a spectral set $\sigma$ is given by
\begin{equation}
  \label{eq:riesz_spectralpoj-def}
  P_{\sigma} := \frac1{2\pi i} \int_{{\Gamma_\sigma}}\resolventop{T}{\lambda} \mathrm{d}\lambda,
\end{equation}
where $\Gamma_\sigma$ is a simple, closed integration path oriented
counterclockwise that lies in the resolvent set $\resolventset{T}$ and
encloses $\sigma$. Recall, that $\lambda$ is a pole of the resolvent of $T$ if
and only if $T - \lambda I$ has positive finite chain length $p$ which also is
the order of the pole. In this case $\lambda \in \specp{T}$, \ie, $\lambda$ is
an eigenvalue of $T$. The spectral projector $P_{\setof{\lambda}}$ corresponding
to $\setof{\lambda}$ satisfies
\begin{equation}
   \label{eq:riesz_spectral-projection-browder}
  \im(P_{\setof{\lambda}}) = \ker(T - \lambda I)^p\quad\text{and}\quad
  \ker(P_{\setof{\lambda}}) = \im(T - \lambda I)^p.
\end{equation}
If furthermore $T - \lambda I$ is a Fredholm operator, i.e. a Browder operator,
then $\lambda$ is always an isolated eigenvalue of $T$ and the associated
spectral projection is finite-dimensional.

Note that the computation of the spectral projection using the formula provided
in \eqref{eq:riesz_spectralpoj-def} is in general hard to calculate. In the next
section, we will consider operators $T$ where $T - \lambda I$ is a Browder
operator and explicitly construct the corresponding spectral projection
$P_{\setof{\lambda}}$.

\section{Invariant subspaces of linear operators}



The aim of this section is to show how to construct a
projection $P$ onto a generalized eigenspace of a bounded linear operator $T$
 defined on a complex Banach space $X$ corresponding to an eigenvalue $\lambda \in \Cc$.
To this end, we consider an operator $T$ such that $T - \lambda I$ is a Browder operator
with ascent $\asc{T - \lambda I} = p$. In this case, the projection has the property 
$\ker(P) = \im(T - \lambda I)^p$ which gives us generically the following space decomposition:
\begin{alignat*}{2}
  X &= \ker(T - \lambda I)^p\quad &&\oplus\quad \im(T - \lambda I)^p\\
  &= \quad \im(P) &&\oplus\quad\quad \ker(P).
\end{alignat*}
We provide a simple criterion under which assumptions this space
decomposition is possible. Before we will look at a finite-dimensional
generalized eigenspace of an operator $T \in \linope{X}$, we will construct the
projection on an arbitrary finite-dimensional subspace $M$ of a vector
space $X$. On $M$ we introduce the classical coordinate map defined by
a basis of $M$ and the corresponding dual basis of the dual space
$M^*$. By the extension theorems of Hahn-Banach the coordinate map
gives us a continuous projection of $X$ onto $M$. In the sequel, we
will discuss conditions on the functionals that can be chosen in the
coordinate map to build a dual basis. Finally, we apply the results to
the generalized eigenspaces of a bounded linear operator $T$ on a
Banach space $X$ and its adjoint $T^*$ corresponding to an eigenvalue
$\lambda \in \Cc$. A necessary condition on the operator
$T - \lambda I$ is being Fredholm with non-positive index. If in
addition $T - \lambda I$ is a Browder operator, \ie, the index is zero
and its chain length is finite, then the projection yields the
previously mentioned direct sum decomposition of $X$.

Note that this space decomposition is already well known, see
\eqref{eq:riesz_spectral-projection-browder}, provided $T - \lambda I$ has
positive finite chain length. In contrast to existing literature we prove it
using an explicitly constructed finite-rank projection $P$. This method uses in
fact the restriction that $T - \lambda I$ has to be Weyl operator, \ie, a
Fredholm operator of index zero, to guarantee that the corresponding generalized
eigenspaces of $T$ and $T^*$ have finite dimension. This direct construction of
the projection $P$ provides an alternative way to calculate the spectral
projection corresponding to the eigenvalue $\lambda$ as by
\eqref{eq:riesz_spectralpoj-def}.

\subsection{Dual basis and the coordinate map}
Let $X$ be a normed vector space over the complex numbers and let $M \subset X$ 
be a closed subspace with $0 < \dim(M) < \infty$. In the sequel, 
we denote its dimension by $n = \dim(M)$.
Moreover, let $\setof{e_1,\ldots,e_n}$ be a basis for $M$. 
Then every $x \in M$ has a unique representation
\begin{equation}
  \label{eq:lincomb_basis}
  x = \sum_{i = 1}^n a^*_i(x)e_i,
\end{equation}
where $\setof{a^*_1,\ldots,a^*_n}$ are appropriate continuous linear
functionals on $M$. By definition, each $u_i$ can also be represented by
\eqref{eq:lincomb_basis} which yields the characterization
\begin{equation}
  \label{eq:mu_ek_identity}
  a^*_i(e_k) = \delta_{ik} =
  \begin{cases}
    1 & \text{if } i = k\\
    0 & \text{if } i \neq k,
  \end{cases}
\end{equation}
for all $i,k \in \setof{1,\ldots, n}$. In analogy to the construction of the
frame operator on Hilbert spaces \cite{christensen2001}, 
we define a \emph{synthesis operator} $\Phi:\Cc^n \to M$ by
\begin{equation}
  \label{eq:synthesis}
  \Phi(a_1,\ldots,a_n) = \sum_{i=1}^n a_i e_i.
\end{equation}
The adjoint of this operator $\Phi^*:M \to \Cc^n$ yields the \emph{analysis} operator
\begin{equation}
  \label{eq:analysis}
  \Phi^*(x) =
  \begin{pmatrix}
    a^*_1(x)\\ \vdots \\ a^*_n(x)
  \end{pmatrix},\qquad x \in X.
\end{equation}
Combining both operators we can represent the coordinate map 
\eqref{eq:lincomb_basis} by the composition $\Phi^*\Phi: M \to M$,
\begin{equation*}
  (\Phi\Phi^*)(x) = \sum_{i=1}^n a^*_i(x) e_i = x.
\end{equation*}
Note that according to \eqref{eq:mu_ek_identity} the matrix
$\Phi^*\Phi \in \Cc^{n\times n}$ is the identity on $\Cc^n$:
\begin{equation*}
  \Phi^*\Phi =
  \begin{pmatrix}
    a^*_1(e_1) & \cdots & a^*_1(e_n)\\
    \vdots & & \vdots\\
    a^*_n(e_1) & \cdots & a^*_n(e_n)    
  \end{pmatrix} = 
  \begin{pmatrix}
    1 & 0 & \cdots & 0\\
    0 & 1 & \cdots & 0\\
    \vdots &  & \ddots&\vdots\\
    0 & \cdots & 0 &1
  \end{pmatrix} = I_n.
\end{equation*}
Accordingly, the basis $\setof{a^*_1,\ldots,a^*_n} \subset M^*$ is
said to be the \emph{dual basis} for $\setof{e_1,\ldots, e_n} \subset M$.
Applying the Theorem of Hahn-Banach, the coordinate map can be
extended to the whole vector space $X$. 

\begin{figure}[tb]
  \centering
  \begin{tikzpicture}[scale=1.3, transform shape]
    \matrix (m) [matrix of math nodes, row sep=4em,
    column sep=3.5em, text height=2ex, text depth=0.25ex]
    { X & & M \\
      \Cc^n & & \Cc^n\\ };
    \path[->,font=\scriptsize]
    (m-1-1) edge node[left] {$\Phi^*$} (m-2-1)
    (m-2-3) edge node[right] {$\Phi$} (m-1-3)
    (m-1-1) edge node[auto] {$\Phi\Phi^*$} (m-1-3)
    (m-2-1) edge[->] node[below] {$(\Phi^*\Phi)^{-1} = I_n$} (m-2-3)
    (m-2-1) edge node[auto] {$\Phi$} (m-1-3);
  \end{tikzpicture}
  \caption{A commutative diagram that illustrates the projection $\Phi\Phi^*:X\to M$ as composition of 
    the synthesis operator $\Phi$ and analysis operator $\Phi^*$. 
   Note that the projection can also be written as $\Phi(\Phi^*\Phi)^{-1}\Phi^*$.}
  \label{fig:it_commutative}
\end{figure}
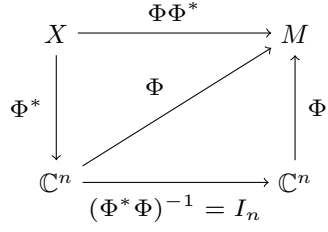

\begin{lemma}
  \label{lem:projection_bounded}
  The operator \(\Phi\Phi^*:M \to M\) can be extended to a projection of the space $X$ 
  onto the closed set $M$  and is bounded by
  \begin{equation*}
    \norm{(\Phi\Phi^*)(x)} \leq \norm{x}\sum_{i=1}^n\norm{a^*_i}.
  \end{equation*}
  The matrix $(\Phi^*\Phi)_{ij} \in \Cc^{n\times n}$ is invertible and 
  the coordinate map \(\Phi\Phi^*\big|_M:M \to M\) which is restricted on $M$
  yields an isomorphism. The space \(X\) can be decomposed into
  \begin{equation*}
    X = M \oplus \ker(\Phi\Phi^*).
  \end{equation*}
\end{lemma}
\begin{proof}
  The continuous functionals $a^*_i$ can be extended by the classical
  Hahn-Banach Theorem to $X^*$ with the same properties as on $M$. We
  denote the resulting extensions again as $a^*_i \in X^*$. Therefore,
  $\Phi\Phi^*:X \to M$ and
  \begin{equation*}
    (\Phi\Phi^*)(x) = \sum_{i=1}^n a^*_i(x) e_i = x\qquad\text{for all }x \in M.
  \end{equation*}
  Moreover, the operator is bounded on $X$ since for $x \in X$ we have
  \begin{equation*}
    \norm{(\Phi\Phi^*)(x)} = \norm{\sum_{i=1}^n a^*_i(x) e_i} \leq \sum_{i=1}^n \norm{a^*_i(x)} \norm{e_i}
    \leq \norm{x} \sum_{i=1}^n \norm{a^*_i},
  \end{equation*}
  where we used that $\norm{e_i} = 1$ and the fact that
  $\norm{a^*_i(x)} \leq \norm{a^*_i}\norm{x}$. Clearly, $(\Phi^*\Phi)$
  is invertible with $(\Phi^*\Phi)^{-1} = I_n$. It yields also a projection, 
  because for every $x \in M$ we obtain $(\Phi\Phi^*)(x) = x$ and
  therefore, $(\Phi\Phi^*)^2 = (\Phi\Phi^*)$. As the operator
  $\Phi\Phi^* \in \linope{X}$ is a bounded projection onto the closed
  space $M$, we obtain canonically the space decomposition
  $X = M \oplus \ker(\Phi\Phi^*)$.
\end{proof}
The key property to notice here is that $(\Phi^*\Phi)_{ij}$ is
an invertible matrix and that $\Phi\Phi^*$ is a projection onto $M$.
The commutative diagram shown in \autoref{fig:it_commutative} illustrates
the behaviour of $\Phi$ and $\Phi^*$.

In the following, we show which functionals can be chosen instead of the dual
basis such that $\Phi\Phi^*$ is still a projection where the analysis operator
$\Phi^*$ now contains the new functionals. The next section shows that the
matrix $\Phi^*\Phi$ must have full column rank.

\subsection{Complemented subspaces and projections}
We consider now the following problem. Given a set of linear functionals
$\Lambda \subset X^*$, we ask whether it is possible to construct a projection
onto the closed finite dimensional subspace $M \subset X$ with functionals
chosen only from the set $\Lambda$. We give a characterization in the next
theorem. As in the previous section, we consider a finite-dimensional subspace
$M$ of $X$. Additionally, let $\Lambda \subset X^*$ be a finite-dimensional
subspace of $X^*$. Let us denote by $\setof{e_1,\ldots,e_n}$ and
$\setof{e^*_1,\ldots,e^*_m}$ a basis of $M$ and $\Lambda$, respectively. The
\emph{synthesis operator} $\Phi:\Cc^n \to X$ is constructed as in
\eqref{eq:synthesis}, whereas the \emph{analysis operator} $\Phi^*:X \to \Cc^n$
is this time not defined as the adjoint of $\Phi$ but uses the basis functionals
of $\Lambda$:
\begin{equation*}
  \Phi^*(x) :=
  \begin{pmatrix}
    e^*_1(x)\\ \vdots \\ e^*_m(x)
  \end{pmatrix},\qquad x \in X.
\end{equation*}
Let us assume that $\dim(\Lambda) \geq \dim(M)$ holds. Then we
will show in the next theorem that again $\Phi\Phi^*$ 
yields a projection operator onto $M$ provided that $\Phi^*\Phi$ has full column rank.

\begin{figure}[tb]
  \centering
  \begin{tikzpicture}[scale=1.3, transform shape]
    \matrix (m) [matrix of math nodes, row sep=4em,
    column sep=3.5em, text height=2ex, text depth=0.25ex]
    { X & & M \\
      \Cc^m & & \Cc^n\\ };
    \path[->,font=\scriptsize]
    (m-1-1) edge node[left] {$\Phi^*$} (m-2-1)
    (m-2-3) edge node[right] {$\Phi$} (m-1-3)
    (m-1-1) edge node[auto] {$\Phi A\Phi^*$} (m-1-3)
    (m-2-1) edge[->] node[auto] {$A$} (m-2-3);
  \end{tikzpicture}
  \caption{Commutative diagram showing the projection $\Phi A\Phi^*:X\to M$. Here
  the matrix $A$ is either the Moore-Penrose inverse of the matrix $\Phi^*\Phi$ or its inverse.}
  \label{fig:it_commutative2}
\end{figure}

\begin{theorem}
  \label{thm:projection_arbitrary_functionals}
  Let $\Lambda \subset X^*$ with $0 < \dim(M) \leq \dim(\Lambda) < \infty$ 
  and let $n = \dim(M)$, $m = \dim(\Lambda)$.
  Then the operator $P \in \linope{X}$ defined for $A = (a_{ij}) \in \Cc^{n \times m}$ by
  \begin{equation}
    \label{eq:projection_op}
    Px  = \Phi A\Phi^*(x) = \sum_{i=1}^{n}\sum_{j=1}^m a_{ij} e^*_j(x) e_i, \qquad x \in X,
  \end{equation}
  yields a projection onto $M$ if and only if the matrix
  \begin{equation}
    \label{eq:gramian_matrix_unknown_functionals}
    G := (\Phi^*\Phi) =
    \begin{pmatrix}
      e^*_1(e_1) & \cdots & e^*_1(e_n)\\
      \vdots & & \vdots\\
      e^*_m(e_1) & \cdots & e^*_m(e_n)
    \end{pmatrix}  \in \Cc^{m \times n}
  \end{equation}
  has full column rank $n$. 
  In this case, the matrix $A$ is determined by the Moore-Penrose inverse of $G$,
  \begin{equation*}
    A = (G^TG)^{-1}G^T. 
  \end{equation*}
\end{theorem}
\begin{proof}
  Let us first assume that for the matrix $G \in \Cc^{m \times n}$ exists the
  Moore-Penrose inverse
  \begin{equation*}
    G_{left}^{-1} = (G^TG)^{-1}G^T \in \Cc^{n \times m},
  \end{equation*}
  and let $A = G_{left}^{-1}$. According to its definition we have
  \begin{equation}
    \label{eq:G_inverse}
    A \cdot G = G_{left}^{-1} \cdot G = I_n.
  \end{equation}
  Now, we prove that $P$, defined for $x \in X$ as in \eqref{eq:projection_op},
  is a projection onto $M$. To this end, we will show that $P(x) = x$ holds for
  all $x \in M$ by considering the basis of $M$.
  Thus, we only have to prove $P(e_k) = e_k$ for all $k \in \setof{1,\ldots,n}$. 
  The direct calculation of $P(e_k)$ yields
  \begin{equation*}
    P(e_k) = \sum_{i=1}^n\sum_{j =1}^m a_{ij} e^*_j(e_k) e_i
           = \sum_{i=1}^n \left[\sum_{j=1}^m a_{ij} e^*_j(e_k)\right] e_i.
  \end{equation*}
  Applying \eqref{eq:G_inverse} yields $\sum_{j=1}^m a_{ij} e^*_j(e_k) = \delta_{ki}$.
  Therefore, 
  \begin{equation*}
    P(e_k) = \sum_{i=1}^n p_i \delta_{ki} = e_k
  \end{equation*}  
  holds and we obtain $P(X) = M$. 
  Furthermore, $P^2 = P$ on $X$ as $\setof{e_1,\ldots,e_n}$ forms a basis for $M$.
  Finally, we show the reverse direction. To this end, 
  let us assume that $P$ is a projection onto $M$, \ie, $P(X) = M$ and $P^2 = P$ holds. 
  Then $P(e_k) = e_k$ must hold for any $k \in \setof{1,\ldots,n}$, as $e_i \in M$. 
  We calculate
  \begin{equation*}
    P(e_k) = \sum_{i=1}^n\sum_{j=1}^m a_{ij} e^*_j(e_k) e_i 
           = \sum_{i=1}^n \left[\sum_{j=1} a_{ij} e^*_j(e_k)\right] e_i.
  \end{equation*}
  This yields necessary the requirement $\sum_{j=1}^m a_{ij} e^*_j(e_k) = \delta_{ik}$ for all 
  $k \in \setof{1,\ldots,n}$. 
  Therefore, we derive the matrix equation $A \cdot G = I_n$ 
  with the unknown coefficient matrix $A = (a_{ij}) \in \Cc^{n \times m}$.
  In fact, this equation has a solution if and only if the matrix $G$ 
  has a Moore-Penrose inverse $G_{left}^{-1}$, which concludes the proof.
\end{proof}
Next, we will provide an upper bound of the projection operator
$P = \Phi A\Phi^*$ by the $1$-norm of $A$.
\begin{lemma}
  \label{lem:P_compact}
  Under the assumption of \autoref{thm:projection_arbitrary_functionals},
  the projection operator $P$ defined by \eqref{eq:projection_op} has finite-rank 
  and is bounded by
  \begin{equation*}
    \norm{Px} \leq \norm{x}\sum_{i=1}^n\sum_{j=1}^m\abs{a_{ij}},\qquad x\in X.
  \end{equation*}
\end{lemma}
\begin{proof}
  Clearly, $P$ is a  finite rank operator. Let $x \in X$.
  For arbitrary $i \in \setof{1,\ldots,n}$ we obtain
  \begin{equation*}
    \norm{\sum_{j=1}^ma_{ij}e^*_j(x)} \leq \norm{x}\sum_{j=1}^m \abs{a_{ij}},
  \end{equation*}
  because the dual basis is normalized, \ie, $\norm{e^*_j} = 1$. Using the same 
  argument for the basis of $M$ we get
  \begin{equation*}
    \norm{Px} = \norm{\Phi A \Phi^*x} = \sum_{i=1}^{n}\sum_{j=1}^m a_{ij} e^*_j(x) e_i
    \leq \norm{x}\sum_{i=1}^n\sum_{j=1}^m\abs{a_{ij}}.
  \end{equation*}
\end{proof}

\subsection{Invariant subspaces and projections}
\label{ssec:it_invariation-subspaces}
In the following we will consider a linear operator $T$ defined on 
a complex Banach space $X$. As in the preceding sections
we are interested in the construction of a projection onto a finite-dimensional
subspace of $X$. Here, we choose $M$ as a generalized eigenspace of
$T$ corresponding to an eigenvalue $\lambda \in \sigma_p(T)$.
It will be shown that the set of functionals is exactly given by
the corresponding generalized eigenspace of the adjoint $T^*$. 

Accordingly, given some integer $p>0$, we consider now the following two subspaces
\begin{align}
  \label{eq:it-inv-fixpoints}
  &M^p_\lambda = \ker(T - \lambda I)^p = \cset{x \in X}{(T - \lambda I)^px = 0} \subset X,\\
  \label{eq:it-inv-fixpointsdual}
  &\Lambda^p_\lambda = \ker(T^* - \lambda I)^p = \cset{x^* \in X^*}{(T^* - \lambda I)^px^* = 0} 
  \subset X^*.
\end{align}
Note that due to the fact that $\ker(T^* - \lambda I)^p = \im((T-\lambda I)^p)^\bot$ holds
by \eqref{eq:prel_relation_annihilator_nullspace} 
the set $\Lambda^p_\lambda$ can also be determined as
\begin{equation}
  \label{eq:it-inv-fixpointsdual2}
  \Lambda^p_\lambda = \cset{x^* \in X^*}{x^*\left((T - \lambda I)^px\right) = 0 \text{ for all }x \in X}.
\end{equation}
To assure that both spaces \eqref{eq:it-inv-fixpoints} and
\eqref{eq:it-inv-fixpointsdual} are finite-dimensional and that the dimension of the 
functionals $\Lambda^p_\lambda$ is greater than the dimension of $M^p_\lambda$, 
we assume in the following that $(T - \lambda I)^p$ is a Fredholm operator with negative
index, \ie, $\ind{T-\lambda I}^p \leq 0$. 
Then we have by definition
\begin{equation*}
  n = \dim(M^p_\lambda) \leq \dim(\Lambda^p_\lambda) = m
\end{equation*}
and we can consider w.l.o.g. normalized bases of $M^P_\lambda$ and $\Lambda^P_\lambda$:
\begin{equation}
  \label{eq:fixed-points-normalized-bases}
  M^p_\lambda = \spanof\setof{e_1,\ldots,e_n} \text{ and } \Lambda^p_\lambda = \spanof\setof{e_1^*,\ldots,e_m^*}
\end{equation}
such that $\norm{e_i}_X = 1$ and $\norm{e^*_i}_{X^*} = 1$.
If we additionally suppose we have the following finite chain of inclusions
\begin{equation*}
  \ker(T - \lambda I) \subsetneq \ker(T - \lambda I)^2 \subsetneq \cdots \subsetneq 
  \ker(T -\lambda I)^p = \ker(T - \lambda I)^{p+1} = \cdots,
\end{equation*}
then the ascent of $T - \lambda I$ is specified as $p := \asc{T - \lambda I} < \infty$.
Corresponding to the eigenvalue $\lambda \in \specp{T}$, the set $M_\lambda$
contains all of the generalized eigenvectors of the operator $T$ and
the set $\Lambda^p_\lambda$ contains all the dual
generalized eigenvectors. More precisely, the set $\Lambda^p_\lambda$
contains all the generalized eigenvectors of the adjoint operator
$T^*$ to the eigenvalue $\lambda$.

\begin{remark}
  Note that the assumption on $T$ are not very restrictive. As shown in the end of the last chapter,
  every compact operator satisfies all of the conditions. Moreover, quasi-compact operators
  satisfy these condition in the case where $\lambda = 1$ is chosen. 
  Especially, every operator where $T - \lambda I$ is a Browder operator fulfills these
  conditions, see the definition and comments in \autoref{ssec:compact_fredholm_riesz}.
\end{remark}

Next, we will show how to construct a projection $P$ onto $\ker(T-\lambda I)^p$ 
to obtain the space decomposition
\begin{equation*}
  \ker(T -\lambda I)^p \oplus \ker(P)
\end{equation*}
such that $\ker(P) = \im(T -\lambda I)^p$ holds. Note that in this
case $\im(T - \lambda I)^p$ is closed as it is the null space of the projection $P$.

First, we provide an equivalent characterization of the restrictions
on $T$  to have finite chain length of the generalized eigenspaces of $T$ 
provided that $T - \lambda I$ is a Fredholm operator with $\ind{T - \lambda I} \leq 0$ 
to assure that the generalized eigenspaces of $T$ and $T^*$ are finite-dimensional. 
The next lemma shows that the ascent can be characterized by the column rank of 
the Gramian matrix constructed using the matrix \eqref{eq:gramian_matrix_unknown_functionals}.
In the following, we will denote by $\fredmope{X}$ all Fredholm operators defined on
the Banach space $X$ that have an index less or equal to zero.
\begin{lemma}
  \label{lem:gram_invertible}
  Let $T \in \linope{X}$ and $\lambda \in \sigma_p(T)$ such that 
  $T -\lambda I \in \fredmope{X}$. Then $T - \lambda I$ has finite ascent $p$, 
  \ie, $p = \asc{T - \lambda I} = p < \infty$,
  if and only if the Gramian matrix
  \begin{equation*}
    G := (\Phi^*\Phi) =
    \begin{pmatrix}
      e^*_1(e_1) & \cdots & e^*_1(e_n)\\
      \vdots & & \vdots\\
      e^*_m(e_1) & \cdots & e^*_m(e_n)
    \end{pmatrix}  \in \Cc^{m \times n}
  \end{equation*}
  has full column rank.
\end{lemma}
\begin{proof}
  Suppose that $T -\lambda I$ is Fredholm operator with non-positive index. 
  Then $(T- \lambda I)^p$ is also Fredholm with non-positive index 
  $p\cdot \ind{T - \lambda I}$. This follows by the index theorem \cite[Thm.~23.1]{heuser1982},
  as 
  \begin{equation*}
    \ind{\underbrace{(T-\lambda I)\cdots(T-\lambda I)}_{p-\text{times}}} = \sum_{i=1}^p\ind{T-\lambda I} = p \cdot \ind{T - \lambda I}.
  \end{equation*}
  Therefore, $\im(T-\lambda I)^p$ is closed \cite[Prop.~24.3]{heuser1982} and 
  \begin{equation*}
    n = \alpha((T - \lambda I)^p) = \dim(\ker(T-\lambda I)^p) \leq \dim(\ker(T^*-\lambda I)^p) 
    = \beta((T - \lambda I)^p) = m.
  \end{equation*}
  Note that $(\Lambda^p_\lambda)_\bot = (\im((T- \lambda I)^p)^\bot)_\bot 
  = \overline{\im(T- \lambda I)^p} = \im(T - \lambda I)^p$.

  Let us now assume that $T - \lambda I$ has ascent $p$. In order to show that
  the columns of $G = \Phi^*\Phi$ are linearly independent, we choose
  $c  = (c_1,\ldots,c_n)^T \in \Cc^n$ such that
  \begin{equation*}
    \sum_{i=1}^n c_i e^*_j(e_i) = 0
  \end{equation*}
  for all $j \in \setof{1,\ldots, m}$. 
  Then we derive that $e^*_j(\sum_{i=0}^n c_i e_i) = 0$ for all $j \in \setof{1,\ldots, m}$.
  Therefore,
  \begin{equation*}
    \sum_{i=0}^nc_i e_i \in \bigcap_{j=1}^m\ker(e^*_j) =  (\Lambda^p_\lambda)_\bot = \im(T- \lambda I)^p.
  \end{equation*}

  As $T - \lambda I$ has finite ascent $p$ we can conclude with \eqref{eq:finite_asc-desc_properties}
  that $\ker(T- \lambda I)^p \cap \im(T- \lambda I)^p = \setof{0}$ holds. As by definition
  also $\sum_{i=1}^n c_i e_i \in \ker(T- \lambda I)^p$ holds we derive that $\sum_{i=1}^n c_i e_i = 0$.
  From the linear independence of $\setof{e_1,\ldots,e_n}$ it follows that $c_1 = \cdots = c_n = 0$.
  Therefore, the matrix $\Phi^*\Phi$ has full rank, as the columns are linearly independent.

  To show that the converse is also true let us suppose that the matrix
  $G$ has full column rank.  Hence, if $\sum_{i=1}^nc_i e^*_j(e_i) =
  0$ holds it follows that every coefficient $c_i = 0$ for all $i \in \setof{1,\ldots, n}$.
  Suppose now that $x \in \ker(T - \lambda I)^p \cap \im(T- \lambda I)^p$. 
  Then $x$ can be written as linear combination $x = \sum_{i=1}^nc_i e_i$ for some
  coefficients $c_i \in \Cc$.
  As $\im(T-\lambda I)^p = (\Lambda^p_\lambda)_\bot$, we obtain for all $j \in \setof{1,\ldots, m}$
  that
  \begin{equation*}
    0 = e^*_j\left(\sum_{i=1}^nc_ie_i\right) = \sum_{i=1}^n c_i e^*_j(e_i).
  \end{equation*}
  We conclude that $c_i = 0$ for all $i \in \setof{1,\ldots, n}$ as
  the matrix $G$ has full column rank. Finally, we have $x = 0$. 
  Therefore, $\ker(T- \lambda I)^p \cap \im(T- \lambda I)^p = \setof{0}$.  
  By \autoref{eq:finite_asc-desc_properties} 
  this is equivalent to the statement that the ascent of $T - \lambda I$ 
  is $p$ and the proof is complete.
\end{proof}
As the Gramian matrix has full column rank, we can construct a projection
operator onto $\ker(T - \lambda I)$ according to
\autoref{thm:projection_arbitrary_functionals}. Consequently, as in the last
section, we consider the finite-rank operator $P \in \compope{X}$ defined for $x
\in X$ by
\begin{equation}
  \label{eq:projection_op2}
  Px  = (\Phi A\Phi^*)(x) = \sum_{i=1}^n\sum_{j=1}^m a_{ij} e^*_j(x) e_i,
\end{equation}
where $e_i \in M_\lambda$, $e^*_j \in \Lambda_\lambda$ are the normalized
bases and $A = (a_{ij}) \in \Cc^{n \times m}$.  
This time, the functionals $e^*_j$ are explicitly chosen as basis of $\ker(T^* - \lambda I)^p$
where the coefficients $a_{ij}$ serve as parameter.
In this setting \autoref{thm:projection_arbitrary_functionals}
yields a projection operator that projects onto the generalized eigenspace $M^p_\lambda$ 
and provides a space decomposition of $X$ into $X = M^p_\lambda \oplus \ker(P)$.
\begin{corollary}
  \label{lem:decomposition}
  Let $T \in \linope{X}$ and $\lambda \in \sigma_p(T)$ such that $T -\lambda I \in \fredmope{X}$
  with ascent $p \in \Nn$.
  Then the linear operator $P \in \compope{X}$ defined for $x \in X$ as
  \begin{equation*}
    Px = \Phi A \Phi^*(x),
  \end{equation*}
  where $A$ is the Moore-Penrose inverse of $(\Phi^*\Phi)$,
  yields a continuous projection onto $M^p_\lambda \subset X$, where 
  $\im(P) = M^p_\lambda = \ker(T-\lambda I)^p$ is a closed subspace.
\end{corollary}
\begin{proof}
  This is a direct consequence of \autoref{lem:projection_bounded} and
  \autoref{lem:gram_invertible}.
\end{proof}
Note that in the current setting, we obtain a projection $P$ where
$\im(P) = \ker(T - \lambda I)^p$ is a $T$-invariant subspace. Accordingly, 
we have the space decomposition
 \begin{equation*}   
    X = \im(P) \oplus \ker(P) = \ker(T - \lambda I)^p \oplus \ker(P).
 \end{equation*}
In the following we are interested when also $\ker(P)$ is invariant with respect to the operator $T$. 
Then we can decompose the operator $T$ into
\begin{equation*} T = 
  \begin{pmatrix}
    J & 0\\
    0 & S\\
  \end{pmatrix} \in \linope{\ker(T-\lambda I)^p\oplus \ker(P)},
\end{equation*}
where $J$ is the Jordan normal form of $T$ on the generalized eigenspace 
$\ker(T-\lambda I)^p$ and $S \in \linope{\ker(P)}$  is equal to the operator $T$ 
restricted to $\ker(P)$. 
\begin{remark}
  Even though we write the operator decomposition in matrix notation, we don't
  assume the Banach space $X$ to be separable. The matrix form is only used to
  simplify notation as $\ker(T - \lambda I)$ is always finite-dimensional. In
  this case, $J$ is given according to some basis, whereas $S$ is not
  necessarily defined by a countable dense set in $X$.
\end{remark}
Furthermore, we are not only interested whether $\ker(P)$ is invariant 
with respect to $T$, we also want to know under which conditions on $T$ the
relation $\ker(P) = \im(T-\lambda I)^p$ holds. It turns out 
that this is exactly the case when the $T - \lambda I$ is a Browder operator,
\ie, the operator $T - \lambda I$ has Fredholm index $0$ and finite
chain length $p$. We will discuss this particular case in the following.
First, we show in the next lemma that the Fredholm index $0$ of
$T - \lambda I$ leads to the invertibility of the Gramian matrix
$\Phi^*\Phi$. Finally, we will prove that in this case
$\ker(P) = \im(T-\lambda I)^p$ holds. We will conclude this section
with an overview over related results.
\begin{lemma}
  Let $T \in \linope{X}$ and $\lambda \in \Cc$ such that $T - \lambda I \in \fredmope{X}$.
  Then $T-\lambda I$ is a Browder operator if and only if 
  the matrix $G = \Phi^*\Phi$ is invertible.
\end{lemma}
\begin{proof}
  If $G$ is invertible, then $T - \lambda I$ has finite ascent $p$
  by \autoref{lem:gram_invertible} and $G$ is necessarily a square matrix, thus
  $\ind{T -\lambda I}^p = 0$ as 
  \begin{equation}
    \label{eq:it_proof_indzero_dimensions}
    n = \alpha((T - \lambda I)^p) = \beta((T - \lambda I)^p) = m,
  \end{equation}
  using the definition of the nullity $\alpha((T - \lambda I)^p) = \dim\ker(T - \lambda I)^p$
  and the deficiency $\beta((T - \lambda I)^p) = \dim\ker(T^* - \lambda I)^p$.
  As  $\alpha(T - \lambda I) = \beta(T - \lambda I)$ and $\asc{T - \lambda I} = p < \infty$,
  we can conclude by \cite[Prop. 38.5 and 38.6]{heuser1982} that also the descent of $T - \lambda I$ is finite. Therefore, $T - \lambda I \in \browdope{X}$,
  \ie, $T - \lambda I$ is a Browder operator with ascent $p$.

  Assume to the contrary that $T - \lambda I$ is a Browder operator.
  Then $T - \lambda I$ has finite ascent $p$ and $\ind{T - \lambda I} = 0$
  by definition. As $\ind{T - \lambda I} = 0$ the matrix $G = \Phi^*\Phi$ is
  a $n \times n$-matrix as $n = \alpha((T - \lambda I)^p) = \beta((T - \lambda I)^p)$
  using the same argument as in \eqref{eq:it_proof_indzero_dimensions}.
  As we have the conditions $\ind{T - \lambda I} = 0$ and $\asc{T - \lambda I} = p$ 
  we can apply \autoref{lem:gram_invertible}   
  to conclude that the matrix $G$ has full rank and, thus, is invertible as a square matrix.
\end{proof}
Next, we will prove that the null space of the projection $P$ is given
by $\ker(P) = \im(T - \lambda I)^p$ provided that $T - \lambda I$ is a
Fredholm operators with index $0$ having finite chain length $p$, \ie,
$T - \lambda I$ is a Browder operator. Note that the invertibility of $\Phi^*\Phi$ 
is already sufficient for this result.
\begin{theorem}[Space decomposition]
  \label{thm:space_decomposition_ker_im}
  Let $T \in \linope{X}$ and $\lambda \in \sigma_p(T)$ 
  such that $T-\lambda I$ is a Browder operator with ascent $p$. Then
  \begin{equation*}
    X = \ker(T-\lambda I)^p \oplus \im(T- \lambda I)^p,
  \end{equation*}
  where $\im(\Phi (\Phi^*\Phi)^{-1} \Phi^*) = \ker(T - \lambda I)^p$ and $\ker(\Phi (\Phi^*\Phi)^{-1} \Phi^*) = \im(T - \lambda I)^p$.
\end{theorem}
\begin{proof}
  Let $n = \dim(\ker(T-\lambda I)^p) = \dim(\ker(T^* - \lambda I)^p) < \infty$.
  As $(T-\lambda I)^p$ is a Fredholm operator, $\im(T - \lambda I)^p$ is closed.
  We already have shown that $\im(P) = \ker(T-\lambda I)^p$.
  In order to show $\ker(P) = \im(T- \lambda I)^p$ let $x \in \ker(P)$. Then we have
  \begin{equation}
    \label{eq:proof_it_basis-linindependent}
   0 =  Px = \sum_{i=1}^{n}\sum_{j=1}^n a_{ij} e^*_j(x) e_i.
  \end{equation}
  As $\setof{e_1,\ldots,e_n}$ form a basis for $\ker(T- \lambda I)^p$ by \eqref{eq:it-inv-fixpoints} and 
  \eqref{eq:fixed-points-normalized-bases}, relation \eqref{eq:proof_it_basis-linindependent}
  can only hold if
  \begin{equation*}
    \sum_{j=1}^n a_{ij} e^*_j(x) = 0
  \end{equation*}
  for every $i \in \setof{1,\ldots, n}$. Using that $A = (\Phi^*\Phi)^{-1}$  
  is invertible by \autoref{lem:gram_invertible}, we obtain that $e^*_j(x) = 0$ 
  for all $j \in \setof{1,\ldots, m}$. Then it is easy to see that 
  \begin{equation*}
    x \in (\Lambda^p_\lambda)_\bot = (\im((T- \lambda I)^p)^\bot)_\bot = \im(T- \lambda I)^p, 
  \end{equation*}
  because $\im(T-\lambda I)^p$ is closed.

  Now let $y \in \im(T-\lambda I)^p$. Accordingly, there is $x \in X$ with $(T - \lambda I)^px = y$. 
  In this case also $y \in \ker(P)$ holds, because
  \begin{equation*}
    Py = \sum_{i=1}^{n}\sum_{j=1}^m a_{ij} e^*_j((T - \lambda I)^px) e_i = 0.
  \end{equation*}
  In the last step we used that $e^*_j \in \im((T - \lambda I)^p)^\bot$.
  Finally, we obtain the space decomposition
  \begin{equation*}
    X = \im(P) \oplus \ker(P) = \ker(T- \lambda I)^p \oplus \im(T-\lambda I)^p,
  \end{equation*}
  where $\im(P) = \ker(T-\lambda I)^p$ and $\ker(P) = \im(T-\lambda I)^p$.
\end{proof}

We conclude this section with a theorem that gathers all the results we have 
shown for a bounded operator $T$ with eigenvalue $\lambda \in \sigma_p(T)$,
where $T - \lambda I$ is a Weyl operator, \ie, a Fredholm operator with zero index.
Note once more that this restriction is important for our setting where the 
generalized eigenspaces have to be finite-dimensional.
\begin{theorem}[Characterization of the Browder operator $T - \lambda I$]
  \label{thm:projection_invariants2}\hfill\\
  Let $T \in \linope{X}$ and $\lambda \in \sigma_p(T)$ 
  such that $T - \lambda I \in \weylope{X}$.
  Then the following statements are equivalent:
  \begin{enumerate}
  \item $T - \lambda I$ is a Browder operator, $T - \lambda I \in \browdope{X}$,
  \item the operator $T - \lambda I$ has finite chain length, \ie, 
    $\asc{T-\lambda I} = \dsc{T - \lambda I} < \infty$,
  \item the space $X$ can be decomposed into $X = \ker(T- \lambda I)^p \oplus \im(T - \lambda I)^p$,
  \item the $n \times n$ matrix $G := (\Phi^*\Phi)$,
    \begin{equation*}
      G =
      \begin{pmatrix}
        e^*_1(e_1) & \cdots & e^*_1(e_n)\\
        \vdots & & \vdots\\
        e^*_n(e_1) & \cdots & e^*_n(e_n)
      \end{pmatrix} \in \Cc^{n \times n}
    \end{equation*}
    is invertible, where $n = \dim\ker(T - \lambda I)^p = \dim\ker(T^* - \lambda I)^p$,
  \item the operator $P:X \to \ker(T-\lambda I)^p$ defined by
    \begin{equation*}
      Px  = \Phi A\Phi^*(x) = \sum_{i=1}^{n}\sum_{j=1}^n a_{ij} e^*_j(x) e_i,\qquad x\in X,
    \end{equation*}
    yields a projection onto $\ker(T- \lambda I)^p$, where $A = (a_{ij}) := G^{-1}$.
  \end{enumerate}
\end{theorem}
The main contribution of this article is the invertibility of the Gram matrix and
the construction of the projection operator as stated in the last item. We have shown
an explicit construction of the projection operator $P$ by the inverse of the Gramian matrix
for the space decomposition in the third item. In the next section, we will
apply these results to uniform ergodic theorems.

\section{Application: Uniform ergodic theorems}
We conclude this article by showing a relation between the theory developed in the last sections 
and uniform ergodic theorems. \citet{Sine:1970} has shown that if $T$ is a
contraction on a Banach space $X$ then the Ces\'aro mean
\begin{equation*}
  a_n(T) := n^{-1}\sum_{k=0}^{n-1}T^k
\end{equation*}
converge strongly for $n \to \infty$ if and only if the fixed points of $T$ separate the fixed points
 of $T^*$. 

We show here that for a contraction $T$ where $T - I$ is a Weyl operator, \ie,
 a Fredholm operator of index $0$, the this fixed point separation property is
 equivalent to the property that $T - I$ has ascent one. This states in
 particular that $T- I$ is in fact a Browder operator.
\begin{theorem}
  Let $T \in \boundope{X}$ such that $\normop{T} \leq 1$ and $T - I \in \weylope{X}$.
  Then $\ker(T - I)$ separates the points of $\ker(T^* - I)$ if and only if the matrix
  \begin{equation*}
    G =
    \begin{pmatrix}
      e^*_1(e_1) & \cdots & e^*_1(e_n)\\
      \vdots & & \vdots\\
      e^*_n(e_1) & \cdots & e^*_n(e_n)
    \end{pmatrix} \in \Cc^{n \times n}
  \end{equation*}
  is invertible, where $n = \dim(T - \lambda I) = \dim(T^* - \lambda I)$.
\end{theorem}
\begin{proof}
  We show first that if the fixed points of $T$ separate the fixed points of $T^*$
  then the matrix $G$ is invertible. To this end, let us assume to the contrary 
  that the matrix $G$ is not invertible. We will show that in this case
  $\ker(T - I)$ does not separate $\ker(T^* - I)$. If $G$ is not invertible, then
  the rows of $G$ are not linearly independent. 
  Hence, we can assume there are $c_1,\ldots,c_n \in \Cc$ such that 
  \begin{equation*}
    \sum_{j=1}^n c_j e_j^*(e_i) = 0,\qquad \text{for all }i \in \setof{1,\ldots,n},
  \end{equation*}
  where there exists at least one coefficient with $c_k \neq 0$. Then
  \begin{equation*}
    e_k^*(x) = \sum_{j\neq k} -\left(\frac{c_j}{c_k}\right) e_j^*(x)
  \end{equation*}
  for all $x \in \ker(T - I)$ as $e_1,\ldots,e_n$ form a basis. We conclude that
  $\ker(T - I)$ does not separate $\ker(T^* - I)$.
  
  We prove next by contradiction that if $G$ is invertible then $\ker(T - I)$ separates $\ker(T^* -I)$.
  To this end, assume that the fixed points of $T$ do not separate the fixed points of $T^*$.
  Then there are $x_1^* \neq x_2^* \in \ker(T^* - I)$ such that for all $x \in \ker(T - I)$
  \begin{equation*}
    x_1^*(x) = x_2^*(x).
  \end{equation*}
  Let $x_1^* = \sum_{j=1}^nc_j e_j^*$ and $x_2^* = \sum_{j=1}^nb_j e_j^*$. Then as well
  \begin{equation*}
    x_1^*(e_i) - x_2^*(e_i) = \sum_{j=1}(c_j - b_j) e_j^*(e_i) = 0
  \end{equation*}
  holds for all $i \in \setof{1,\ldots,n}$. 
  As $c_i \neq b_i$ for at least one $i \in \setof{1,\ldots,n}$ the rows of $G$ are linearly dependent
  and $G$ is not invertible.
\end{proof}

Finally, we extend our results of \autoref{thm:projection_invariants2} with the result of the 
previous theorem.
\begin{corollary}
  Let $T \in \boundope{X}$ with $\normop{T} \leq 1$ such that $T - I \in \weylope{X}$.
  Then the following statements are equivalent:
  \begin{enumerate}
  \item $T - I$ has chain length one, \ie, $\asc{T - I} = \dsc{T - I} = 1$,
  \item $X = \ker(T - I) \oplus \im(T - I)$,
  \item $T - I \in \browdope{X}$,
  \item $G$ is invertible,
  \item $P =\Phi G^{-1} \Phi^*$ yields a projection onto $\ker(T - I)$,
  \item The Ces\'aro means $n^{-1}\sum_{k=0}^{n-1}T^k$ converge in the strong operator topology towards $P$
    for $n \to\infty$.
  \end{enumerate}
\end{corollary}
The last item follows in particular by work of \cite[Thm.~3.16 on p.~215]{Dunford:1943}.



\section{Iterates of quasi-compact Markov operators}
\label{sec:iterates}
Using the preceeding results, we consider now the limit of the iterates of an
operator $T \in \boundope{X}$ that has a non-trivial fixed point space.
We will first introduce the concept of quasi-compact operators in a
proper way and relate the quasi-compactness to the essential spectrum
and the Browder essential spectrum. Note that if the iterates converge
to a finite-rank operator, then this operator is quasi-compact by
definition. Results for the peripheral spectrum of quasi-compact
operators are given with corollaries for the case where the operator
is positive.  Finally, we will state different limit theorems for
quasi-compact operators.

\subsection{Quasi-compact operators and the peripheral spectrum}
For Banach spaces there are several ways to define the essential spectrum for a bounded
linear operator $T \in \linope{X}$. If one considers
the essential spectrum as the largest subset of the spectrum which remains invariant
under compact perturbations one obtains the following definition of $\sigma_{ess}(T)$,
\begin{equation*}
  \sigma_{ess}(T)  := \cset{\lambda \in \Cc}{T - \lambda I \not \in \weylope{X}},
\end{equation*}
which also often said to be the \emph{essential Weyl spectrum}, see \citet{schechter1966}. 
However, this definition of the 
spectrum does not contain the limit points of the spectrum. 
If all these accumulation points are added, then one comes to the definition of 
\citet[107]{browder1961}, where a spectral value $\lambda \in \Cc$ 
is in the essential spectrum, if at least one of the following conditions hold:
\begin{enumerate}
\item $\im(T - \lambda I)$ is not closed in $X$,
\item $\lambda$ is a limit point of the spectrum $\sigma(T)$,
\item $\bigcup_{k \in \Nn}\ker(T - \lambda I)^k$ is infinite dimensional.
\end{enumerate}
This is indeed equivalent to the \emph{essential Browder spectrum},
\begin{equation*}
  \sigma_{b}(T)  = \cset{\lambda \in \Cc}{T - \lambda I \not \in \browdope{X}}.
\end{equation*}
The advantage of using $\sigma_{ess}$ is the perturbation invariance, while the advantage
of the Browder spectrum $\sigma_{b}(T)$ is that $\sigma(T) \setminus \sigma_{b}(T)$ 
is a countable set. Summing up these facts, we have the relation
\begin{equation*}
  \sigma_{ess}(T) \subseteq \sigma_b(T) = \sigma_{ess}(T) \cup \mathrm{acc}\,\sigma(T)  \subset \sigma(T),
\end{equation*}
where $\mathrm{acc}\,\sigma(T)$ denotes all the limit points of $\sigma(T)$.  Nevertheless,
the essential spectral radius is in both definition of the essential spectrum equal, \ie, all 
spectral limit points are on the boundary of $\sigma_{ess}(T)$.

Now, suppose $T \in \linope{X}$ is a quasi-compact operator, \ie, 
the essential spectral radius is less than one. From this it follows that
every spectral value $\lambda \in \sigma(T)$ with modulus larger than the essential spectral radius
is an isolated eigenvalue and the operator $T - \lambda I$ is a Browder operator.
Therefore, there always exists an eigenvalue $\lambda \in \sigma(T)$ 
with modulus equal to the spectral radius $r(T)$.  Moreover, there are only finitely many eigenvalues 
on the peripheral spectrum. The next lemma gives a characterization.
\begin{lemma}
  \label{lem:quasicomp_browder}
  Let $T \in \linope{X}$ be a quasi-compact operator with $r(T) \geq 1$. 
  Then, there is at least one eigenvalue $\lambda$ with $\abs{\lambda} = r(T)$.
  Besides, every spectral value $\lambda \in \sigma(T)$ with $\abs{\lambda} > r_{ess}(T) $ is
  an isolated eigenvalue of $T$ and   $T - \lambda I$ is a Browder operator.
  There are only finitely many eigenvalues on the peripheral spectrum of $T$.  
\end{lemma}
\begin{proof}
  By the definition of quasi-compactness, we have $r_{ess}(T) < 1$ and
  all of the spectral values outside with modulus larger than
  $r_{ess}(T)$ are isolated.  As already discussed above, $T - \lambda
  I$ is a Browder operator. If $\lambda \not\in \sigma_b(T)$, then
  by Theorem 1 in \cite{lay1968}, $\lambda$ is a pole of the resolvent of finite rank.
  Applying \citet[Proposition 50.3]{heuser1982}, 
  we derive that $\asc{T - \lambda I}$ is
  positive and hence, $\lambda$ is an eigenvalue of $T$.  As all the
  cluster points of the spectrum are on the boundary of
  $\sigma_{ess}(T)$, there is an eigenvalue $\lambda \in \sigma(T)$
  with $\abs{\lambda} = 1$.
  Finally, there are only finitely many on the peripheral spectrum as otherwise there
  would be an accumulation point outside of the essential Browder spectrum.
\end{proof}
We now show that for the eigenvalues $\lambda$ of quasi-compact operators $T \in \linope{X}$ 
lying on the peripheral spectrum, \ie, eigenvalues $\lambda$ with modulus $r(T)$, the associated
Browder operator $T - \lambda I$ has ascent one, whereas the dimension of the associated eigenspace
can be arbitrary but finite. This result has already been shown in a similar setting by 
\citet[Proposition V.1]{hennion2001} and is stated here with less assumptions on the operator $T$.
\begin{lemma}
  \label{lem:periph_ascent_one}
  Let $T \in \linope{X}$ be a quasi-compact operator with $r(T) \geq 1$ such that 
  $\sup_{n \in \Nn}r(T)^{-n}\normop{T^n} < \infty$. Then for every peripheral eigenvalue $\lambda \in \sigma_{per}(T)$ the associated Browder operator $T- \lambda I$ has ascent one.
\end{lemma}
\begin{proof}
  Note that the existence of a peripheral eigenvalue has been shown in the lemma above.
  We consider now the eigenvalue $\lambda \in \sigma_p(T)$ with $\abs{\lambda} = r(T)$. 
  Let $x \in \ker(T - \lambda I)^2$. Then we can represent $T^nx$ for all positive integers $n$ by
  \begin{equation*}
    T^nx = (\lambda I + (T - \lambda I))^nx = \sum_{k=0}^n\binom{n}{k}\lambda^{n-k} (T - \lambda I)^kx
    = \lambda^nx - n \lambda^{n-1} (T - \lambda I)x.
  \end{equation*}
  We will show now that $(T - \lambda I)x = 0$. To this end, using that
  $\abs{\lambda} = r(T)$ and that there exists of $B > 0$ such that
  $r(T)^{-m}\normop{T^m} < B$ for all positive integers $m$, we calculate:
  \begin{equation*}
    \norm{n \lambda^{n-1} (T - \lambda I)x} = \norm{\lambda^n x - T^nx} \leq \abs{\lambda^n}\norm{x} + \norm{T^nx}
    \leq r(T)^n\norm{x} + \normop{T^n}\norm{x} \leq r(T)^n(1 + B)\norm{x}.
  \end{equation*}
  It is now easy to see that $\norm{(T - \lambda I)x} \leq \frac{r(T)(1 + B)}{n}\norm{x}$
  and we conclude that $x \in \ker(T - \lambda I)$ as $n$ was arbitrary.
\end{proof}

In the following, we consider the case when $r(T) = \normop{T}$. Operators of this kind are said to be
\emph{normaloid} and have been discussed in \citet[Chapter 54]{heuser1982}. We obtain the following corollary:
\begin{corollary}
  \label{cor:quasicomp_browder_acent_one}
  Let $T \in \linope{X}$ be a quasi-compact operator with $r(T) = \normop{T}$.
  The exists at least one eigenvalue with modulus $r(T)$. Furthermore, 
  for every peripheral eigenvalue $T - \lambda I$ is a Browder operator with ascent one.
\end{corollary}
\begin{proof}
  For all positive integers $n$ the inequality $r(T) \leq \normop{T^n}^{1/n} \leq \normop{T}$ holds.
  Therefore, $r(T)^{-n}\normop{T^n} \leq 1$ for all $n$. The result follows by 
  \autoref{lem:quasicomp_browder} and \autoref{lem:periph_ascent_one}.
\end{proof}

If we $X$ is a Banach lattice and $T$ is positive even stronger results can be made.
According to \citeauthor{lotz1968}, \citet{krein1948} 
have first shown that every positive compact operator on a Banach lattice with 
$r(T) = \normop{T}$ has a cyclic peripheral spectrum.
This result has been generalized in \citet[Theorem 4.10]{lotz1968}, where the peripheral spectrum
 of a positive operator $T \in \linope{X}$  on a Banach lattice $X$ is cyclic if $r(T)$ 
is a pole of the resolvent. 
Furthermore, \citet{lotz1968} concluded that the peripheral spectrum
of every positive compact operator is cyclic.

The next corollary sums up these results for positive quasi-compact operators.
\begin{corollary}
  \label{cor:quasicomp_browder_acent_one}
  Let $T \in \linope{X}$ be a positive quasi-compact operator with $r(T) = \normop{T} = 1$.
  Then $1 \in \sigma(T)$, \ie, $T - I$ is a Browder operator of ascent one.
  Furthermore, the peripheral spectrum is cyclic and consists only of roots of unity.
\end{corollary}
\begin{proof}
  It has been shown by \citet{lotz1968} that $r(T) \in \sigma(T)$ if $T$ is positive. 
  In the case where $T$ is a quasi-compact positive operator with $r(T) = 1$,  
  $T$ has real eigenvalue one and $T - I$ is a Browder operator with ascent one.

  The peripheral spectrum contains only finitely many eigenvalues of $T$ and $1 \in \sigma_{per}(T)$.
  As the peripheral spectrum is cyclic, see the above mentioned result in \cite{lotz1968}, we conclude
  that $\sigma_{per}(T)$ can only contain roots of unity.
\end{proof}

If $T \in \linope{X}$ is a positive quasi-compact operator on a Banach lattice $X$ with
$\normop{T} = r(T) = 1$. Then of course $\sigma(T) \subset \uballc$ and using the preceding 
results we obtain that $1$ is an isolated eigenvalue of $T$ and the peripheral spectrum is cyclic.
Let us denote by the positive integer $l$ the number of spectral values in the peripheral spectrum.
 There are now two cases to discuss separately: 
\begin{enumerate}
\item $l = 1$: then $\sigma_{per}(T) = \setof{1}$, otherwise
\item $\sigma_{per}(T) = \cset{\rho_l^k}{k \in \setof{1,\ldots, l}}$,
\end{enumerate}
where $\rho_l$ are the $l$-th roots of unity.

The first case has already been characterized by \citet{katznelson1986}, who have
been shown that for every linear operator $T$ on a Banach space $X$ with $\normop{T} \leq 1$
the limit
\begin{equation*}
\lim_{n\to\infty}\normop{T^{n+1} - T^n} = \lim_{n\to\infty}\normop{T^{n}(T - I)} = 0 
\end{equation*}
holds if (and only if) $\sigma_{per}(T) \subset \setof{1}$.

\subsection{Operators with ascent one}
The last results have shown that if $\lambda$ is a peripheral eigenvalue of
quasi-compact operator $T$, the operator $T - \lambda I$ has always ascent one.
In this case the spaces $M^1_\lambda$ and $\Lambda^1_\lambda$ contain only
eigenvectors of $T$ and $T^*$ respectively. They can now be represented by
\begin{align}
  \label{eq:fixpoints2}
  &M^1_\lambda = \ker(T - \lambda I) = \cset{x \in X}{Tx  = \lambda x},\\
  \label{eq:fixpointsdual2}
  &\Lambda^1_\lambda = \ker(T^* - \lambda I) = \cset{x^* \in X^*}{ x^*(Tx) = x^*(\lambda x) \text{ for all } x \in X}.
\end{align}
Let us denote by $n = \dim(M^1_\lambda) = \dim(\Lambda^1_\lambda)$.
The result of \autoref{thm:projection_invariants2} yields a projection $P:X \to M^1_\lambda$ onto 
the eigenspace space associated with $\lambda$. Recall that the Gram matrix
consisting of only of dual eigenvectors acting on the eigenvectors of $T$,
\begin{equation*}
  G := (\Phi^*\Phi) =
  \begin{pmatrix}
    e^*_1(e_1) & \cdots & e^*_1(e_n)\\
    \vdots & & \vdots\\
    e^*_m(e_1) & \cdots & e^*_m(e_n)
  \end{pmatrix}  \in \Cc^{n \times n},
\end{equation*}
is invertible. Setting $A = (a_{ij}) = G^{-1}$, the projection has the form 
\begin{equation*}
  Px = \Phi(\Phi^*\Phi)^{-1}\Phi^*(x) = \sum_{i=1}^n\sum_{j=1}^m a_{ij} e^*_j(x) e_i, \qquad \text{for } x \in X.
\end{equation*}
The next lemma gives a characterization of this projection.
\begin{theorem}
  \label{thm:ascent_one_projection}
  Let $T \in \linope{X}$ and $\lambda \in \sigma_p(T)$ such that $T - \lambda I$ is a Browder operator.
  Then the following two statements are equivalent:
  \begin{enumerate}
  \item $T - \lambda I$ has ascent one,
  \item there exists a projection $P \in \compope{X}$ such that $TP = PT = \lambda P$,
  \end{enumerate}
\end{theorem}
\begin{proof}
  Suppose that the first statement holds. Then we obtain a projection $P \in \compope{X}$ 
  from \autoref{thm:projection_invariants2}. We have the property $T \circ P = \lambda P$, 
  because for $x \in X$ we obtain
  \begin{equation*}
    (T \circ P)(x) = T(\sum_{i=1}^n\sum_{j=1}^m a_{ij} e^*_j(x) e_i) 
    = \sum_{i=1}^n\sum_{j=1}^m a_{ij} e^*_j(x) T(e_i) = \lambda P(x),
  \end{equation*}
  as $Te_i = \lambda e_i$ by \eqref{eq:fixpoints2}. Similarly, we obtain $P \circ T = \lambda P$.
  Namely, for $x \in X$ using $e^*_j(Tx) = e^*_j(\lambda x)$ by \eqref{eq:fixpointsdual2} it holds that
  \begin{equation*}
    (P\circ T)(x) = \sum_{i=1}^n\sum_{j=1}^m a_{ij} e^*_j(Tx) e_i 
    = \sum_{i=1}^n \sum_{j=1}^m a_{ij} e^*_j(\lambda x) e_i = \lambda P(x).
  \end{equation*}

  Now we show that if there exists a projection $P \in \compope{X}$ with $TP = PT = \lambda P$,
  then $T - \lambda I$ has ascent one.  
  As $\ker(T - \lambda I) \subset \ker(T - \lambda I)^2$ and $\im(P) = \ker(T - \lambda I)$, 
  it is enough to show that $\ker(T - \lambda I)^2 \subset \im(P)$.
  Suppose $x \in \ker(T - \lambda I)^2$. Then $(T - \lambda I)^2x = 0$ and
  $(T - \lambda I)x \in \ker(T - \lambda I) = \im(P)$. Therefore, there is $y \in \im(P)$ such that
  $Py = (T - \lambda I)x$. Then
  \begin{equation*}
    y = Py = P^2y = P(T - \lambda I)x = PTx - \lambda Px = \lambda Px - \lambda Px = 0.
  \end{equation*}
  Thus, $y = 0$ and we obtain using $0 = y = Py = (T - \lambda I)x$
  the final result, namely $x \in \ker(T - \lambda I) = \im(P)$.
\end{proof}

\begin{lemma}
  \label{lem:powers_ascent_one}
  Let $T \in \linope{X}$ and $\lambda \in \sigma_p(T)$. Suppose there exists a projection 
  $P \in \compope{X}$ such that $TP = PT = \lambda P$. Then
  \begin{equation*}
    (T - TP)^n = T^n - T^nP = T^n - \lambda^n P
  \end{equation*}
  holds for all $n \in \Nn$.
\end{lemma}
\begin{proof}
  We will use the fact that if $P$ is a projection then $I - P$ is a projection as well.
  Also note that $I - P$ commutes with $T$, as
  \begin{equation*}
    (I - P)T = T - TP = T - PT = T(I - P).
  \end{equation*}
  Now we derive the result with the following steps:
  \begin{align*}
    (T - TP)^n &= (T - TP)^n = (T(I - P))^n = T^n(I - P)^n\\
    &= T^n(I - P) = T^n - T^nP = T^n - \lambda^n P.
  \end{align*}
\end{proof}

\subsection{The limit of the iterates of quasi-compact operators}
We assume in the following that $\normop{T} = 1$ and $r(T) = 1$. First , we will
restrict us to the fixed point space of a quasi-compact operator $T \in \linope$
and assume that $\sigma(T) \subset \uballo \cup \setof{1}$, \ie, $1$ is the only
peripheral eigenvalue of $T$. In this case, if $T - I \in \browdope{X}$ has
ascent one and the iterates will converge to the projection $P$. Later we will
consider the case where the peripheral spectrum is cyclic.

In the proof of our main result we will need the following lemma. As it is more convenient, 
we omit the proof here but prove it at the end of this section. The lemma states that
isolated spectral values can be removed by the projection operator on the corresponding generalized
eigenspace.
\begin{lemma}
  \label{lem:spectral_removal}
  Let $T \in \linope{X}$ and $\lambda \in \sigma_p(T)$ such that 
  $T - \lambda I \in \browdope{X}$ with ascent $p$.
  Let $P$ be denote the projection onto $\ker(T - \lambda I)^p$, defined by 
  \autoref{thm:projection_invariants2}. 
  Then $\lambda$ is an isolated spectral value and $\lambda \notin \sigma(T - TP)$.
\end{lemma}

Now we can show the following result:
\begin{theorem}
  \label{thm:iterates_limit}
  Let $T \in \boundope{X}$ with $r(T) = \normop{T} = 1$ satisfying the spectral condition
  $\sigma(T) \subset \uballo \cup \setof{1}$.
  Then $T$ is quasi-compact if and only if
  \begin{equation*}
    \lim_{m \to \infty} \normop{T^m - P} = 0,
  \end{equation*}
  where $P \in \compope{X}$ is a finite-rank projection with $TP = PT = P$.
\end{theorem}
\begin{proof}
  Clearly, if the iterates $T^m$ converge to a finite-rank operator,
  then $T$ is a quasi-compact operator.

  Now let $T$ be quasi-compact with $\sigma(T) \subset \uballo \cup \setof{1}$, 
  then $r(T) = 1$ and $1$ is an isolated peripheral eigenvalue. 
  Thus, $T - I$ is Browder with ascent one. We now prove the limit of the iterates.
  By \autoref{thm:space_decomposition_ker_im} the space $X$ has the decomposition
  \begin{equation}
    \label{eq:it_space_decomposition}
    X = \ker(T - I) \oplus \im(T- I).
  \end{equation} 
  Therefore, we can decompose the operator $T$ into
  \begin{equation*}
    T =
    \begin{pmatrix}
      I & 0\\
      0 & S
    \end{pmatrix} \in \linope{\ker(T - I) \oplus \im(T- I)},
  \end{equation*}
  with $S \in \linope{\im(T-I)}$. Then
  \begin{equation*}
    T - P =
    \begin{pmatrix}
      0 & 0\\
      0 & S
    \end{pmatrix},
  \end{equation*}
  where $P$ is the projection operator defined by \autoref{thm:projection_invariants2}.
  Using \autoref{lem:spectral_removal} we derive that 
  $1 \not\in \sigma(S) \subset \uballo \cup \setof{1}$, and hence, 
  $\sigma(S) \subset \uballo$. Therefore, the spectral radius of $S$ is strictly smaller than $1$
  and thus, the iterates $S^m$ converge to $0$ in the operator norm as $m$ tends to infinity. 
  Finally, applying \autoref{lem:powers_ascent_one} we obtain the final result
  \begin{equation*}
    \lim_{m \to \infty}\normop{T^m - P} = \lim_{m \to \infty}\normop{(T - P)^m} \lim_{m \to \infty}\normop{S^m} = 0.
  \end{equation*}
  The iterates $T^m$ converge in the strong topology to the operator $P$, 
  the projection onto the fixpoint space of $T$.
\end{proof}

\begin{corollary}[Convergence Rate]
  \label{cor:iterates_convergence_rate}
  Let $T \in \boundope{X}$ be a quasi-compact operator 
  with $r(T) = \normop{T} = 1$ satisfying the spectral condition $\sigma(T) \subset \uballo \cup \setof{1}$. Define
  \begin{equation*}
    \gamma := \sup\cset{\abs{\gamma}}{\gamma \in \sigma(T) \setminus \setof{1}}.
  \end{equation*}
  Then there exists a constant $1 \leq C \leq \gamma^{-1}$, such that for all $m \in \Nn$
  \begin{equation*}
     \normop{T^m - P} \leq C \cdot \gamma^m,
  \end{equation*}
  where $P \in \compope{X}$ is the operator defined by \autoref{thm:projection_invariants2}.
\end{corollary}
\begin{proof}
  According to the proof of \autoref{thm:iterates_limit} we decompose 
  \begin{equation*}
    T = 
    \begin{pmatrix}
      I & 0\\
      0 & S
    \end{pmatrix} \in \boundope{\ker(T - I) \oplus \im(T - I)}.
  \end{equation*}
  Furthermore, we have that $\sigma(S) \subset \uballo$ and therefore we obtain
  $r(S) = \gamma < 1$. As $r(S) = \lim_{m\to \infty}\norm{S^m}^{1/m}$, we obtain that
  there exists a constant $1 \leq C \leq \gamma^{-1}$ such that
  \begin{equation*}
    \norm{S^m} \leq C \cdot \gamma^m
  \end{equation*}
  for every $m \in \Nn$.
\end{proof}

If a sequences of operators with the spectrum contained in $\uballo \cup \setof{1}$ 
share the same fixpoints spaces, the following limit theorem hold.
\begin{corollary}
  Let $T_n \in \boundope{X}$ be a sequence of continuous linear operators with 
  $\sigma(T_n) \subset \uballo \cup \setof{1}$ such that $T_n - \lambda I \in \browdope{X}$ 
  with ascent one. Furthermore, assume that $\ker(T_n - I)$ and $\ker(T^*_n - I)$ 
  are equal for all $n \in \Nn$.
  If $(k_n)_{n\in\Nn} \subset \Nn$ is a strictly increasing sequence of positive integers, then
  \begin{equation*}
    \lim_{n\to\infty}\normop{T_n^{{k_n}} - P} = 0,
  \end{equation*}
  where $P \in \compope{X}$ is the operator defined by \autoref{thm:projection_invariants2}.
\end{corollary}
\begin{proof}
  Follows directly by applying \autoref{cor:iterates_convergence_rate}.
  We derive for $n \in \Nn$ that
  \begin{equation*}
    \norm{T_n^{{k_n}} - P} = \norm{S_n^{{k_n}}} \leq C \gamma_n^{{k_n}}
  \end{equation*}
  where
  \begin{equation*}
    \gamma_n := \sup\cset{\abs{\gamma}}{\gamma \in \sigma(T_n) \setminus \setof{1}}.
  \end{equation*}
  As $1$ is an isolated eigenvalue, $\gamma_n < 1$ for all $n \in \Nn$ and therefore,
  $\norm{T_n^{k_n} - P} \to 0$ if $n$ tends to infinity.
\end{proof}

Finally we discuss the case, when the peripheral spectrum is cyclic.
\begin{theorem}
  \label{thm:iterates_limit_cyclic}
  Let $T \in \boundope{X}$ be a quasi-compact operator with $r(T) = \normop{T} = 1$
  with a non-trivial fixed point space.
  Furthermore, we assume the peripheral spectrum to be finite and cyclic.
  Then there exists $k \in \Nn$ such that
  \begin{equation*}
    \lim_{m \to \infty} \normop{T^{km} - P} = 0,
  \end{equation*}
  where $P \in \compope{X}$ is the operator defined by \autoref{thm:projection_invariants2} for 
  applied to the operator $T^k$.
\end{theorem}
\begin{proof}
  As the peripheral spectrum is finite and cyclic and $1 \in \sigma_{per}(T)$,
  the spectrum contains only roots of unity. Let us denote by $l$ the number of 
  spectral values contained in the spectrum. Then
  \begin{equation*}
    \sigma_{per}(T) = \cset{\rho_l^k}{k \in \setof{1,\ldots,l}},
  \end{equation*}
  where $\rho_l$ is the $l$-th root of unity.
  By the spectral mapping theorem for the point spectrum, 
  see \eg, \citet[Theorem 10.33]{rudin1991} we conclude that for the integer
  $k := l(l-1)\cdots2\cdot1$ the peripheral spectrum of $T^k$ contains only the eigenvalue $1$.
  As $T^k$ is also quasi-compact, we can derive the result by \autoref{thm:iterates_limit}
  applied to $T^k$.
\end{proof}

It is easy to see that quasi-compact Markov operators always satisfy the conditions of
\autoref{thm:iterates_limit_cyclic}. For that case, we derive the corollary:
\begin{corollary}
  Let $T \in \linope{X}$ be a quasi-compact Markov operator with $\normop{T} = 1$. 
  Then there exists $k \in \Nn$ such that
  \begin{equation*}
    \lim_{m \to \infty} \normop{T^{km} - P} = 0,
  \end{equation*}
  where $P \in \compope{X}$ is the operator defined by \autoref{thm:projection_invariants2} for 
  applied to the operator $T^k$.
\end{corollary}

\section*{References}
\bibliographystyle{plainnat}
\bibliography{references}

\end{document}